\documentclass[oneside,english,reqno]{amsart}

\usepackage[latin1]{inputenc} 
\usepackage{cmap}

\usepackage{amssymb}
\usepackage[]{epsf,epsfig,amsmath,amssymb,amsfonts,latexsym}
\usepackage{comment}
\usepackage{enumerate}
\usepackage{subcaption}
\usepackage[normalem]{ulem}
\usepackage{color}
\usepackage{cancel}
\usepackage{amsmath}
\usepackage[normalem]{ulem}
\usepackage{amsfonts}
\usepackage{bbm}
\usepackage{amsthm}
\usepackage[mathscr]{eucal}
\usepackage{graphicx}
\usepackage{pstricks}
\usepackage[active]{srcltx}
\usepackage{fancyhdr}
\usepackage{verbatim}
\usepackage{fullpage}
\usepackage{hyperref}
\usepackage{mathtools}
\hypersetup{
	colorlinks = true, 
	urlcolor = blue, 
	linkcolor = red, 
	citecolor = blue 
}
\usepackage{bbm}
\usepackage{caption}
\usepackage{float}
\usepackage{tikz}
\usetikzlibrary{arrows,decorations.markings}
\usetikzlibrary{calc}

\usepackage{enumitem}

\newcommand{\Z}{\mathbb{Z}}

\def \P{\mathbb P}

\def \E{\mathbb{E}}
\def \R{\mathbb R}

\def \L{\mathcal L}

\DeclareMathOperator{\var}{Var}
\newcommand{\zero}{{\mathbf 0}}
\newcommand{\extB}{\partial_{\circ}}

\newtheorem{thm}{Theorem}[section]

\newtheorem{prop}[thm]{Proposition}
\newtheorem{cor}[thm]{Corollary}

\newtheorem{lemma}[thm]{Lemma}

\newtheorem{remark}[thm]{Remark}

\newcommand{\eqd}{\,{\buildrel d \over =}\,}

\newcommand{\Ze}{\Z_{\mathrm{even}}}
\newcommand{\Zo}{\Z_{\mathrm{odd}}}

\newcommand{\Cf}{\mathcal{C}_{\mathrm{fin}}}

\newcommand{\eps}{\varepsilon}

\title{Delocalization of uniform graph homomorphisms from $\Z^2$ to $\Z$}

\author{Nishant Chandgotia, Ron Peled, Scott Sheffield \and Martin Tassy}
\address{Centre for Applicable Mathematics, Tata Institute of Fundamental Research}
\email{nishant.chandgotia@gmail.com}
\address{School of Mathematical Sciences, Tel Aviv University}
\email{peledron@post.tau.ac.il}
\address{Department of Mathematics, Massachusetts Institute of Technology}
\email{sheffield@math.mit.edu}
\address{Department of Mathematics, Dartmouth College}
\email{mtassy@math.dartmouth.edu}

\begin{document}
\begin{abstract}
Graph homomorphisms from the $\Z^d$ lattice to $\Z$ are functions on $\Z^d$ whose gradients equal one in absolute value. These functions are the height functions corresponding to proper $3$-colorings of $\Z^d$ and, in two dimensions, corresponding to the 6-vertex model (square ice). We consider the uniform model, obtained by sampling uniformly such a graph homomorphism subject to boundary conditions. Our main result is that the model delocalizes in two dimensions, having no translation-invariant Gibbs measures. Additional results are obtained in higher dimensions and include the fact that every Gibbs measure which is ergodic under even translations is extremal and that these Gibbs measures are stochastically ordered.
\end{abstract}
	\maketitle

    \section{Introduction}
    In this paper we study \emph{homomorphism height functions} on the $\Z^d$ lattice: functions $f:\Z^d\to\Z$ restricted to satisfy $|f(u) - f(v)| = 1$ when $u$ is adjacent to $v$ and taking even values on the even sublattice. Such functions have received special attention in the literature. In one dimension they are the possible trajectories of simple random walk. In two dimensions they are the height functions of the $6$-vertex model (square ice), while in all dimensions they are in bijection with proper $3$-colorings of the lattice by taking the modulo~$3$ operation (after fixing the value at one vertex). There is no single name associated to these height functions in the literature and they have at times been called random graph homomorphisms into $\Z$~\cite{benjamini2000random} (or $\Z$-homomorphisms~\cite{peled2013lipschitz}), $\Z^d$-indexed random walk~\cite{MR1856513}, Body-Centered Solid-On-Solid (BCSOS)~\cite{van1977exactly} or the height functions of the 6-vertex model~\cite{van1977exactly}.

    Our main concern is with the fluctuations exhibited by a homomorphism height function which is uniformly sampled in a domain with zero boundary conditions. The one-dimensional function is a random walk bridge, which fluctuates as the square-root of the length of the domain. In high dimensions the surface localizes, having bounded variance at each vertex, and this is conjectured to occur for all $d\ge 3$~\cite{MR1856513, galvin2003homomorphisms, galvin2015phase, peled2017high}. A lower bound on the typical range of the surface is proved in~\cite{benjamini2007random}. The two-dimensional case is especially interesting, as integer-valued height functions in two dimensions generally undergo a \emph{roughening transition} from a rough (delocalized) phase to a smooth (localized) phase as the parameters of the model are changed~\cite{velenik2006localization}.

     The uniform model on homomorphism height functions has no parameters and hence its behavior is {\em a priori} unclear in two dimensions. However, noting that the model is a special case of the $6$-vertex model (when the $6$ weights are equal) suggests that it delocalizes, as it falls in the disordered regime of the predicted phase diagram~\cite{baxter2007exactly}. Indeed, it is particularly natural to consider the \emph{F-model}, a one-parameter subfamily of the $6$-vertex model, indexed by $c>0$, whose description in terms of height functions is the following: A homomorphism height function $f$ in a domain (with boundary conditions) is sampled with probability proportional to $c^{-N(f)}$ with $N(f)$ being the number of diagonally adjacent vertices $u,v$ on which $f(u)\neq f(v)$. The uniform model that we study here is then obtained for $c=1$, whereas it is predicted that the height function of the F-model delocalizes for all $c\le 2$ and localizes when $c>2$. It has recently been shown that this prediction is correct for $c \geq 2$ ~\cite{duminil2016discontinuity, glazmanpeled2018} and our current paper shows that it is correct for $c = 1$.  Even more recently (following the posting of this article) the prediction has been extended to the range $(1,2)$  \cite{duminil-copin_delocalization_2020}. We further remark that delocalization at $c=1$ is suggested by  rapid mixing results which have been proved for the uniform proper $3$-coloring model~\cite{MR1857394,MR2068870}. In this work we prove that the two-dimensional uniform homomorphism height function delocalizes. Additional results are obtained in higher dimensions.

     \begin{figure}
	\begin{center}
		\includegraphics[scale=0.48]{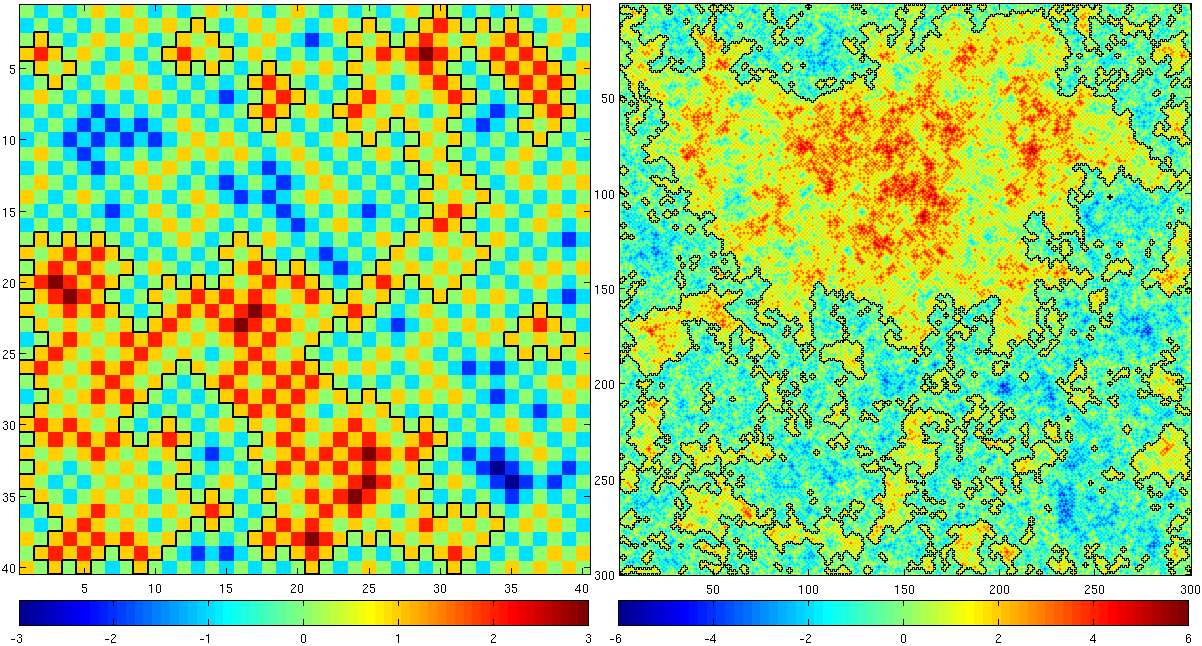}
	\end{center}
	\caption{the outermost level sets separating zeros and ones of a uniform homomorphism height function with zero boundary values (on every second vertex) on a $40\times40$ and $300\times300$ squares, sampled using coupling from the past~\cite{propp1996exact}. Theorem~\ref{thm:main} shows that the height at the center of the squares diverges as the square size increases.}
	\label{fig:hom_with_level_line}
\end{figure}
\begin{figure}
	\begin{center}
		\includegraphics[scale=0.63]{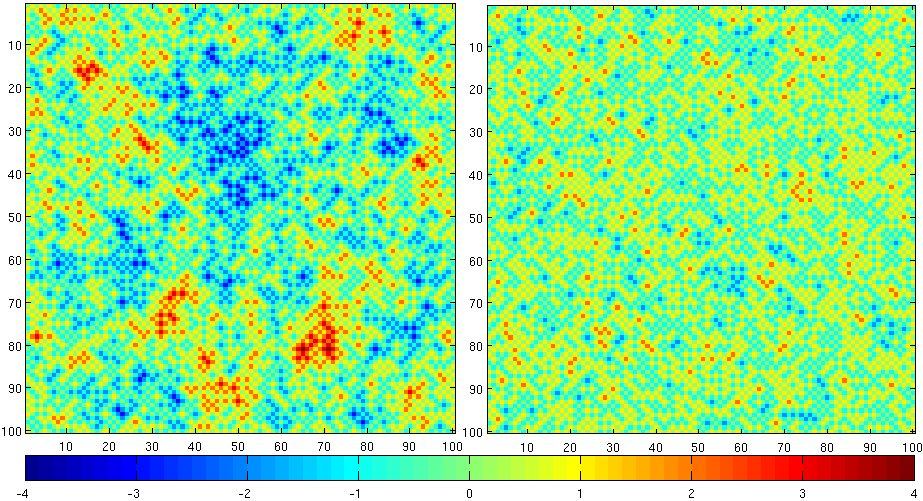}
	\end{center}
	\caption{Simulation of uniform homomorphism height functions with zero boundary conditions (on every second vertex) on a $100\times 100$ square (left) and on the middle slice of a $100\times100\times100$ cube. Sampled using coupling from the past~\cite{propp1996exact}. Unlike the two-dimensional case, it is conjectured that homomorphism height functions in three and higher dimensions are localized. This is known in sufficiently high dimensions~\cite{galvin2015phase, peled2017high}.}
	\label{fig:hom_2d_3d}
\end{figure}

      We note that the exact entropy constant for the two-dimensional uniform model has been found by Lieb~\cite{lieb1967residual}, and the result has been extended to the $6$-vertex parameter space; see~\cite{duminil2016discontinuity,duminil2018bethe, lieb1972two} for related works and surveys. However, these results do not seem to suffice to describe the fluctuations of the height function.

    It is worth mentioning that there are only a handful of results in the mathematical literature where an integer-valued height function has been shown to delocalize in two dimensions. The breakthrough work of Fr\"ohlich--Spencer \cite{MR634447} established this fact for the high-temperature integer-valued Gaussian free field and Solid-On-Solid models, en route to proving the existence of the Kosterlitz-Thouless transition in the plane-rotator (XY) model (see also \cite{kharash2017fr}). {Some } other examples concern integrable models either in the sense that the free energy is exactly calculable (as done by Lieb~\cite{lieb1967residual} for the model studied here) or having an exactly predicted phase diagram, though, as mentioned above, these facts by themselves do not yield a description of the fluctuations. Here is a list of such examples known {(at the time when this paper was written)}.
 \begin{itemize}
\item
The dimer model~\cite{MR1872739, MR1933589} and the uniform spanning tree (for certain domains in $\Z^2$~\cite{MR3781449, russkikh2018dominos}, the honeycomb lattice~\cite{MR2415464} and for planar graphs with some natural restrictions~\cite{berestycki2016universality}).
\item
The $F$-model at the parameter $c=2$~\cite{MR3592746, glazmanpeled2018}, and around the free fermion point $c=\sqrt 2$~\cite{dubedat2011exact,MR3637384}.
\item
The Lipschitz height function model on the triangular lattice (the loop $O(2)$ model) at the parameter $x=1/\sqrt{2}$~\cite{duminil2017macroscopic} and $x=1$~\cite{glazmanManolescu2018}.
\end{itemize}
Since the posting of our paper, delocalisation has been proved for several other models: \cite{lammers_height_2020} (for certain integer-valued height functions on graphs with degree $3$ or less), \cite{lammers2021delocalisation}(for the solid-on-solid model on planar graphs with arbitrary slope) and \cite{duminil-copin_delocalization_2020} (for the six-vertex model when $a=b=1$ and $1< c\leq 2$). Very recently Peled and Glazman provided alternative arguments for the delocalisation of the homomorphism height functions using triangular lattice contours \cite{glazmanpeled2018}.

    As a final perspective on our results, we note that the set of homomorphism height functions form a dynamical system under translations by the even sublattice and has the structure of a countable-state shift of finite type. Our results imply that this dynamical system in two dimensions does not have a measure of maximal entropy. A characterisation
of one-dimensional countable-state shifts of finite type
(also known as topological Markov shifts) which have a measure of maximal entropy can be found in~\cite{MR1955261}; however such a characterisation is not known in higher dimensions.

    \subsection{Notation}
    \subsubsection{$\Z^d$ Lattice} We work throughout on the integer lattice $\Z^d$, $d\ge 1$. Denote $\zero:=(0,0,\ldots, 0)\in\Z^d$. For $v = (v_1, v_2,\ldots, v_d)\in\Z^d$ let $\|v\|_1 := |v_1| + \cdots + |v_d|$. Vertices $v,w\in\Z^d$ are adjacent, denoted $v\sim w$, if $\|v-w\|_1=1$. We give names to the two sublattices in $\Z^d$ defined by the parity of the sum of coordinates,
    \begin{align*}
      \Ze^d &:= \{v\in\Z^d\colon \|v\|_1\text{ is even}\},\\
      \Zo^d &:= \{v\in\Z^d\colon \|v\|_1\text{ is odd}\}.
    \end{align*}
    The (graph) ball of radius $L$ is
    \begin{equation*}
      \Lambda(L):=\{v\in\Z^d\colon \|v\|_1 \le L\},\quad \text{$L\ge 0$ integer}.
    \end{equation*}
    The external vertex boundary of a set $\Lambda\subset\Z^d$ is denoted
    \begin{equation*}
      \extB\Lambda:=\{v\in\Z^d\colon v\notin \Lambda, \exists w\sim v, w\in\Lambda\}
    \end{equation*}
    and the $1$-extension of the set is given by
    \begin{equation*}
      \Lambda^+ := \Lambda\cup\extB\Lambda.
    \end{equation*}

    \subsubsection{Homomorphism height functions}
    An integer-valued function $f$ defined on a subset of $\Z^d$ is called a \emph{homomorphism height function} (or $\Z$-homomorphism) if it satisfies
    \begin{alignat}{2}
      &\text{Lipschitz condition: }\quad&& |f(u) - f(v)|\le \|u-v\|_1,\label{eq:Lipschitz_constraint}\\
      &\text{Parity condition: }&&f(v)\equiv\|v\|_1\pmod 2 \label{eq:parity_constraint}
    \end{alignat}
    for $u,v$ in the domain of $f$. Given \eqref{eq:parity_constraint}, on connected domains condition~\eqref{eq:Lipschitz_constraint} is equivalent to
    \begin{equation}\label{eq:gradients_are_one}
      |f(u) - f(v)| = 1\quad \text{for all adjacent $u,v$ in the domain of $f$.}
    \end{equation}
    The name homomorphism height function is motivated by the fact that a function satisfying~\eqref{eq:gradients_are_one} is a graph homomorphism from its domain to $\Z$. One may check that a homomorphism height function on any domain may be extended to a homomorphism height function on the whole of $\Z^d$, see, e.g.,~\cite[Kirszbraun property]{chandgotia2017kirszbraun} or \cite{MT2017}.

    The restriction of a function $f$ to a subset $\Lambda$ of its domain is denoted $f_{\Lambda}$.

    \subsubsection{Finite-volume Gibbs measures}
    Given a non-empty, finite $\Lambda\subset\Z^d$ and homomorphism height function $\tau:\extB\Lambda\to\Z$, the probability measure $\mu^{\Lambda, \tau}$, termed the Gibbs measure in $\Lambda$ with boundary condition $\tau$, is the uniform measure on homomorphism height functions $f:\Lambda^+\to\Z$ which satisfy that $f \equiv \tau$ on $\extB\Lambda$.

    We write $\mu^{\Lambda,0}$ to indicate that $\tau$ is the identically zero function.

    \subsubsection{Infinite-volume Gibbs measures}
    A probability measure $\mu$ on homomorphism height functions on $\Z^d$ is a \emph{Gibbs measure} if it satisfies the following condition: Sampling $f$ from $\mu$, for any non-empty, finite $\Lambda\subset\Z^d$ the conditional distribution of the restriction $f_{\Lambda^+}$, conditioned on $f_{\Lambda^c}$, equals $\mu^{\Lambda,\tau}$ with $\tau := f_{\extB\Lambda}$, almost surely. One checks easily that if the condition in the definition holds for a set $\Lambda$ then it holds with $\Lambda$ replaced by a subset of it. Hence the definition will be unaltered if we require both $\Lambda$ and $\Lambda^c$ to be connected subsets.

    A sequence of finite-volume Gibbs measures $\mu^{\Lambda_n,\tau_n}$, with $\Lambda_n$ increasing to $\Z^d$, is said to \emph{converge in distribution} (i.e., in the thermodynamic limit) if their marginal distribution on any finite subset of $\Z^d$ (defined for large $n$) converges in distribution. It is standard that the limiting measure is then a Gibbs measure \cite[Theorem 14.15]{MR2807681}.

    A \emph{sublattice} $\L$ of $\Z^d$ is the set of linear combinations with integer coefficients of a collection of linearly independent (over the rationals) vectors $v_1, \ldots, v_k\in\Z^d$. The sublattice has \emph{full rank} if $k=d$.

    As we have defined homomorphism height functions to take even values on $\Ze^d$, it is not possible for a Gibbs measure to be invariant in distribution to all translations of $\Z^d$. It is thus useful to write that a Gibbs measure $\mu$ is \emph{$\L$-translation-invariant}, for a sublattice $\L\subset\Z^d$, if it is invariant in distribution to all translations in $\L$. Automatically, we have that $\L\subset\Ze^d$ and we typically take $\L = \Ze^d$. Correspondingly, $\mu$ is called \emph{$\L$-ergodic} if it is $\L$-translation-invariant and every event $A$ which is invariant to translations in $\L$ satisfies $\mu(A)\in\{0,1\}$.

    A Gibbs measure $\mu$ is called \emph{extremal} if it satisfies that any event $A$ in the tail-sigma-algebra has $\mu(A)\in\{0,1\}$. Here, the event $A$ is in the tail-sigma-algebra if its occurrence for a homomorphism height function $f$ may be determined solely from the restriction $f_{\Lambda^c}$, for any finite subset $\Lambda$. It is standard that an $\L$-translation-invariant extremal Gibbs measure is also $\L$-ergodic (\cite[Theorem 14.15]{MR2807681}).

    \subsection{Results}
    Our first result is a simple dichotomy holding in every dimension between delocalization and localization of homomorphism height functions with zero boundary conditions.
    \begin{thm}\label{thm:simple_dichotomy}
      In each dimension $d\ge 1$ exactly one of the following alternatives holds:
      \begin{enumerate}
        \item[(i)] For odd integers $L\ge 1$, let $f_L$ be sampled from $\mu^{\Lambda(L),0}$. Then the set of random variables $(f_L(\zero))$ is not tight. Precisely,
        \begin{equation}\label{eq:non-tightness}
            \lim_{M\to\infty}\sup_{L} \P(|f_L(\zero)|>M)>0.
        \end{equation}
        \item[(ii)] The sequence of finite-volume Gibbs measures $(\mu^{\Lambda(L),0})$, $L\ge 1$ odd, converges in distribution to a $\Ze^d$-translation-invariant Gibbs measure.
      \end{enumerate}
    \end{thm}
    The next theorem is the main result of this paper.
    \begin{thm}\label{thm:main}
      In dimension $d=2$, there are no $\Ze^2$-translation-invariant Gibbs measures. In particular, two-dimensional homomorphism height functions delocalize in the sense that alternative (i) of Theorem~\ref{thm:simple_dichotomy} holds when $d=2$.
    \end{thm}
    The techniques used in the proof of Theorem~\ref{thm:main} also provide information on the set of Gibbs measures in higher dimensions.
    \begin{thm}\label{thm:ergodic_Gibbs_measures_are_extremal}
      In all dimensions $d\ge 3$: Suppose $\mu$ is an $\L$-ergodic Gibbs measure for a full-rank sublattice $\L$. Then $\mu$ is $\Ze^d$-ergodic, extremal and invariant under interchange of coordinates and reflections in coordinate hyperplanes. In addition, if $\mu_1, \mu_2$ are $\Ze^d$-ergodic Gibbs measures then either $\mu_1$ stochastically dominates $\mu_2$ or $\mu_2$ stochastically dominates $\mu_1$ (with respect to the pointwise partial order on $\Z$-valued functions on $\Z^d$).
    \end{thm}


    We make several remarks regarding the theorems.

    The proof of the dichotomy in Theorem~\ref{thm:simple_dichotomy} makes use of a positive association (FKG) property for the \emph{absolute value} of homomorphism height functions shown by Benjamini, H\"aggstr\"om and Mossel~\cite{benjamini2000random}. The proof implies that the alternative satisfied in each dimension, as well as the limiting measure in the second alternative, remain unchanged if one replaces the domains $(\Lambda(L))$, $L\ge 1$ odd, by another sequence of subsets increasing to $\Z^d$ and having their external vertex boundary on the even sublattice.

    It is known that the single-site marginals of homomorphism height functions have a log-concave distribution on the integers (\cite[Lemma 8.2.4]{MR2251117} and \cite[Proposition 2.1]{MR1856513}). With this fact several {\em a priori} different notions of delocalization coincide. Specifically, for $f_L$ sampled from $\mu^{\Lambda(L),0}$ as in Theorem~\ref{thm:simple_dichotomy} the non-tightness condition~\eqref{eq:non-tightness} is equivalent to
    \begin{equation}\label{eq:no_exponential_moment}
      \sup_{L} \E(\exp(\alpha\cdot |f_L(\zero)|))=\infty\quad \text{for all $\alpha>0$}
    \end{equation}
and also to
    \begin{equation}\label{eq:non-tightness2}
        \sup_{L} \P(|f_L(\zero)|>M)=1\quad\text{for all $M$}.
    \end{equation}
   Noting that $\E(f_L(\zero))=0$ as $f_L\eqd -f_L$, we see that the above conditions are also equivalent to having unbounded variance, $\sup_L \var(f_L(\zero))=\infty$. These ideas are explained in more detail in Section~\ref{section: Log-concavity}.

    Delocalization also holds for one-dimensional homomorphism height functions in the sense that there are no $\Ze$-translation-invariant Gibbs measures and alternative~(i) of Theorem~\ref{thm:simple_dichotomy} holds when $d=1$. This is straightforward to prove using the ergodic theorem and the properties of simple random walk bridge. We thus focus throughout on dimensions $d\ge 2$. As mentioned, it is conjectured that alternative~(ii) of Theorem~\ref{thm:simple_dichotomy} holds for all $d\ge 3$ and this is known in sufficiently high dimensions~\cite{peled2017high,galvin2015phase}.


 In his Ph.D. thesis \cite{MR2251117}, the third author considered a broad class of random surface models, allowing nearest-neighbor interactions with arbitrary convex potentials whose heights take values either in the reals or in the integers, and obtained general results, including a variational principle and large deviation results for these models and a detailed discussion of the set of gradient Gibbs measures of the model with reference to their roughness or smoothness properties. The results obtained there imply and are more general than the properties of Gibbs measures given in Theorem~\ref{thm:ergodic_Gibbs_measures_are_extremal} and Theorem~\ref{thm:uniqueness_up_to_additive_constant} below.

One major contribution of this paper, however, is that we provide new proofs of these results from \cite{MR2251117}, which apply in less generality than those in \cite{MR2251117}, but which are significantly more accessible and concise. Since we focus specifically on graph homomorphisms to $\Z$, we are able to make many simplifications, allowing for a relatively short presentation. We hope these new proofs have value, both on their own merits and as a gateway to the more involved discussion in \cite{MR2251117}. In addition, the main result of this work, Theorem~\ref{thm:main}, requires a new argument. To briefly summarize the difference between our arguments and those in \cite{MR2251117}:
\begin{enumerate}
\item  The proof of Lemma \ref{lem:infinite_cluster_swap} given here and the analogous proof given in \cite[Chapter 8]{MR2251117} both use a form of ``cluster swapping'' in which one begins with a pair $(f,g)$ of height functions and then modifies the pair to create a new pair $(\tilde f, \tilde g)$. In this paper, the values of $f$ and $g$ are simply ``swapped'' on the set $\{v : f(v)>g(v) \}$. In \cite[Chapter 8]{MR2251117}, heights are allowed to be either real valued or integer valued (and nearest neighbor height differences can assume an unbounded range of values, unlike this paper where they are restricted to $\pm 1$) and because of this the procedure for determining the swapping set in \cite[Chapter 8]{MR2251117} ends up being more complicated.
\item The proof of Theorem \ref{thm: no coexistence} given here and the analogous proof given in \cite[Chapter 9]{MR2251117} both use FKG arguments, in a certain two-dimensional setting, to rule out the simultaneous co-existence of an infinite connected cluster and an infinite connected complementary cluster. However, \cite[Chapter 9]{MR2251117} considers gradient Gibbs measures with general slope, which lack rotation and reflection invariance (the only assumption is invariance w.r.t.\ translations by elements of {\em some} sublattice), and coping with the possible lack of symmetry necessitates some fairly involved arguments (involving, for example, the homotopy group of the countably punctured plane; see also \cite[Theorem 1.5]{DCRT2017} for a different approach to a similar issue). In this paper, an additional argument (Corollary~\ref{cor: invariance under permutation of coordinates and reflections}) is used to conclude that the zero slope measure has rotation and reflection invariance and these extra symmetries allow us to deduce Theorem \ref{thm: no coexistence} directly from Zhang's argument.
\item Delocalization of zero slope gradient Gibbs measures is not established in \cite{MR2251117} and a new argument is introduced in this paper, utilizing specific properties of the model under consideration, to prove the main result, Theorem~\ref{thm:main}. This argument can be formulated in a simple way (see Section~\ref{sec:no translation invariant measure}) and we expect it to be useful in other settings as well.
\end{enumerate}
Lastly, the following dichotomy theorem is proved in~\cite{DCHLRR2019}, significantly extending the $d=2$ case of Theorem~\ref{thm:simple_dichotomy}. In two dimensions, when $f_L$ is sampled from $\mu^{\Lambda(L),0}$ for $L$ odd, the variance of $f_L(\zero)$ either grows logarithmically with $L$ or stays bounded as $L$ tends to infinity. Theorem~\ref{thm:main} rules out the second alternative and thus logarithmic growth follows.


    \subsection{Acknowledgements} The work of NC was supported in part by the European Research Council starting grant 678520 (LocalOrder) and ISF grant nos.\ 1289/17, 1702/17, 1570/17, 2095/15 and 2919/19 and was carried out when he was a postdoctoral fellow at Tel Aviv University and the Hebrew Univeristy of Jerusalem. The work of RP was supported in part by Israel Science Foundation grant 861/15 and the European Research Council starting grant 678520 (LocalOrder). The work of SS was supported in part by NSF grants DMS 1209044 and DMS 1712862.

    We thank Alexander Glazman and Yinon Spinka for reading our draft and suggesting various improvements. Steven M. Heilman helped prepare Figure~\ref{fig:hom_with_level_line}. We thank the anonymous referee for multiple suggestions for improving the clarity of the arguments.

    \section{Localization dichotomy}
    In this section we will prove Theorem~\ref{thm:simple_dichotomy}. Let $d\ge 1$. For odd integers $L\ge 1$, let $f_L$ be sampled from $\mu^{\Lambda(L),0}$. It is clear that if alternative~(ii) of the theorem holds, i.e., the finite-volume Gibbs measures~$(\mu^{\Lambda(L),0})$, $L\ge 1$ odd, converge in distribution, then alternative~(i) of the theorem does not hold, i.e.,
    \begin{equation}\label{eq:tightness_at_zero}
      \text{the sequence of random variables $(f_L(\zero))$, $L\ge 1$ odd, is tight.}
    \end{equation}
    Thus to prove Theorem~\ref{thm:simple_dichotomy} we need only show that~\eqref{eq:tightness_at_zero} implies alternative~(ii).

    We rely on the following positive association (FKG) property for the \emph{absolute value} of random homomorphism height functions, which follows immediately from a result proven by Benjamini, H\"aggstr\"om and Mossel~\cite{benjamini2000random}. Positive association of the values themselves also holds (Lemma~\ref{lem:FKG_for_pointwise_partial_order}), and will be used later.
    \begin{prop}\label{prop:FKG_for_absolute_value}
      Let $\Lambda\subset\Z^d$ be non-empty and finite, with $\extB\Lambda\subset\Ze^d$. If $f$ is randomly sampled from $\mu^{\Lambda,0}$ then $|f|$ has positive association, in the sense that if $\varphi,\psi:\{0,1,2,\ldots\}^{\Lambda^+}\to[0,\infty)$ are non-decreasing (in the pointwise partial order on integer-valued functions) one has
      \begin{equation*}
        \E(\varphi(|f|)\psi(|f|))\ge \E(\varphi(|f|))\E(\psi(|f|)).
      \end{equation*}
    \end{prop}
    In \cite[Proposition 2.3]{benjamini2000random} the result stated above has been proved for homomorphism height functions on general finite bipartite graphs, normalized to be $0$ at a fixed vertex. Proposition~\ref{prop:FKG_for_absolute_value} follows by identifying the vertices of $\extB \Lambda$ to a single vertex.

    Our use of the proposition is through the following consequence: If $\Lambda_1\subset\Lambda_2$ are non-empty, finite subsets of $\Z^d$ such that $\extB\Lambda_1, \extB\Lambda_2\subset\Ze^d$ then, when $g_1, g_2$ are randomly sampled from $\mu^{\Lambda_1,0}, \mu^{\Lambda_2, 0}$ respectively, one has that
    \begin{equation}\label{equation:dominates}
        \text{$|g_2|_{\Lambda_1^+}$ stochastically dominates $|g_1|$.}
    \end{equation}
    To prove this we apply Proposition~\ref{prop:FKG_for_absolute_value} with $f = g_2$, $\varphi(|f|):=1-1_{\{f(v)=0\text{ for all $v\in\extB\Lambda_1$}\}}$ (with $1_A$ denoting the indicator function of $A$) and $\psi$ an increasing function which depends only on the restriction of $f$ to $\Lambda_1^+$. Thereby we obtain
    $$\E(\psi(|g_2|))- \E(\psi(|g_1|))\P(g_2(v)=0\text{ for all $v\in\extB\Lambda_1$})\geq \E(\psi(|g_2|))(1-\P(g_2(v)=0\text{ for all $v\in\extB\Lambda_1$}))$$
    which implies \eqref{equation:dominates} after rearrangement.

    We may now derive the following lemma.
    \begin{lemma}
      Let $(\Lambda_n)$ be a sequence of non-empty, finite subsets of $\Z^d$ which increase to $\Z^d$ and satisfy $\extB\Lambda_n\subset\Ze^d$ for all $n$. If~\eqref{eq:tightness_at_zero} holds then the measures $\mu^{\Lambda_n,0}$ converge in distribution.
    \end{lemma}
    \begin{proof}
      Let $g_n$ be sampled from $\mu^{\Lambda_n,0}$. Property~\eqref{equation:dominates} implies that the distributions of $|g_n|$ form a stochastically increasing sequence of measures. Convergence in distribution of $|g_n|$ will then follow by verifying that for each $v\in\Z^d$, the set of random variables $g_n(v)$ is tight (for $n\ge n_0(v)$ with $n_0(v)$ the smallest value for which $v\in\Lambda_{n_0(v)}$). This last property is in turn a consequence of the tightness assumption~\eqref{eq:tightness_at_zero} and of \eqref{equation:dominates}: For $v\in\Ze^d$, by~\eqref{equation:dominates}, the distribution of $|g_n(v)|$ is stochastically bounded by the distribution of $|f_L(\zero)|$ where $L$ is chosen so that $\Lambda_n\subset\Lambda(L)+v$. For $v\notin\Ze^d$ the distribution of $|g_n(v)|$ is stochastically bounded by the distribution of $|g_n(w)|+1$ for a neighbor $w$ of $v$, by the definition of homomorphism height function.

      We have shown that the random functions $|g_n|$ converge in distribution. Let us deduce the same for the $g_n$ themselves. This is done by exhibiting a coupling of the $g_n$ in which $g_n(v)$ converges almost surely for each $v$.

      First, as the distributions of $|g_n|$ stochastically increase to a limit, there exists a coupling of the different $|g_n|$ so that $|g_n(v)|$ is non-decreasing for each $v$ and converges almost surely. We assume that the $|g_n|$ are coupled in this way. Second, to extend this to a coupling of the $g_n$, introduce a collection $(\eps_v)_{v\in\Z^d}$ of uniform and independent signs $\pm 1$ and fix a natural-numbers ordering of $\Z^d$. Note that for each $n$ the distribution of $g_n$ conditioned on $|g_n|$ has the following structure: In each connected component of the set $\{v \colon |g_n(v)|> 0\}$ the values of $g_n$ equal the values of $|g_n|$ times a random sign, with the signs being uniform and independent between different components. We couple the conditional distributions of $g_n$ given $|g_n|$ by deciding that the sign of each component equals $\eps_v$ for the vertex $v$ in the component which is lowest in the fixed order. As $|g_n|$ are non-decreasing, this choice ensures that for each $v$ the sign of $g_n(v)$ will eventually stabilize as $n$ increases, and thus that $g_n(v)$ converges almost surely as we wanted to show.
    \end{proof}
    Now assume the tightness condition~\eqref{eq:tightness_at_zero}. It follows that for any sequence of subsets $(\Lambda_n)$ as in the lemma, the sequence of finite-volume Gibbs measures $(\mu^{\Lambda_n,0})$ converges in distribution. As is standard, the limiting measure $\mu$ is a Gibbs measure. To see that it is the same for all choices of $(\Lambda_n)$ one may take two such sequences of subsets and interlace them (omitting elements as necessary to obtain an increasing sequence). This implies that $\mu$ is $\Ze^d$-translation-invariant, so that alternative (ii) of the theorem holds.

    \section{No translation-invariant Gibbs measures}\label{sec:no translation invariant measure}
    Our proof of Theorem~\ref{thm:main} is based on the following uniqueness statement for ergodic Gibbs measures, which is adapted from ~\cite[Section 9]{MR2251117}. The uniqueness statement is proved in the following sections, where Theorem~\ref{thm:ergodic_Gibbs_measures_are_extremal} is further deduced.
    \begin{thm}\label{thm:uniqueness_up_to_additive_constant}
      In dimension $d=2$: Let $\mu, \mu'$ be $\Ze^2$-ergodic Gibbs measures. Then there is an integer $k$ and a coupling of $\mu,\mu'$ such that if $(\tilde{f},\tilde{g})$ are sampled from the coupling then, almost surely,
      \begin{equation}\label{eq:f_equal_g_plus_2k}
        \tilde{f} = \tilde{g} + 2k.
      \end{equation}
    \end{thm}
    Theorem~\ref{thm:main} is derived from this statement with the following additional argument. Suppose, in order to obtain a contradiction, that $\mu$ is a $\Ze^2$-translation-invariant Gibbs measure. By replacing with an element of its ergodic decomposition, if necessary, we assume that $\mu$ is $\Ze^2$-ergodic (recalling that the $\Ze^2$-ergodic decomposition of a $\Ze^2$-translation-invariant Gibbs measure consists of $\Ze^2$-ergodic Gibbs measures \cite[Theorem 14.15]{MR2807681}). Let $f$ be randomly sampled from $\mu$. Define a homomorphism height function $g$ on $\Z^2$ by
    \begin{equation}\label{eq:g_from_f}
      g(v) := f(-v + (1,0)) - 1.
    \end{equation}
    One checks in a straightforward way that the distribution of $g$ is also a $\Ze^2$-ergodic Gibbs measure, which we denote by $\mu'$. Thus, by Theorem~\ref{thm:uniqueness_up_to_additive_constant}, there exists an integer $k$ and a coupling of $\mu, \mu'$ such that, when sampling $(\tilde{f},\tilde{g})$ from this coupling, the equality \eqref{eq:f_equal_g_plus_2k} holds almost surely. Continuing, we observe that \eqref{eq:g_from_f} implies that
    \begin{equation*}
      f((0,0)) + f((1,0)) = g((1,0)) + g((0,0)) + 2.
    \end{equation*}
    Thus, the equality in distribution
    \begin{equation*}
      \tilde{f}((0,0)) + \tilde{f}((1,0)) \eqd \tilde{g}((0,0)) + \tilde{g}((1,0)) + 2
    \end{equation*}
    also holds.
    However, by \eqref{eq:f_equal_g_plus_2k},
    \begin{equation*}
      \tilde{f}((0,0)) + \tilde{f}((1,0)) = \tilde{g}((0,0)) + \tilde{g}((1,0)) + 4k.
    \end{equation*}
    This implies that $4k = 2$, which contradicts the fact that $k$ is an integer. The contradiction establishes Theorem~\ref{thm:main}.

    \section{Tools}
    In this section we detail some well-known tools which will be used in the proof of Theorem~\ref{thm:uniqueness_up_to_additive_constant}. All of the described results apply in general dimension $d$ except the coexistence results of Section~\ref{sec:no_coexistence} which rely on planarity.

    \subsection{Positive association for pointwise partial order}
    The maximum and minimum of two homomorphism height functions are still homomorphism height functions. Thus the following property is an immediate consequence of the standard FKG inequality \cite{fortuin1971correlation} (see \cite[Proposition~2.2]{benjamini2000random}), with its second part derived from the reverse martingale convergence theorem.
     \begin{lemma}\label{lem:FKG_for_pointwise_partial_order}
       Let $\Lambda\subset\Z^d$ be a non-empty finite subset and $\tau:\extB\Lambda\to\Z$ be a homomorphism height function. If $f$ is randomly sampled from $\mu^{\Lambda,\tau}$ then $f$ has positive association, in the sense that if $\varphi,\psi:(\Z)^{\Lambda^+}\to[0,\infty)$ are non-decreasing (in the pointwise partial order on integer-valued functions) one has
      \begin{equation}\label{eq:positive_association_for_values}
        \E(\varphi(f)\psi(f))\ge \E(\varphi(f))\E(\psi(f)).
      \end{equation}
      Consequently, if $\mu$ is an \emph{extremal} Gibbs measure and $f$ is sampled from $\mu$ then \eqref{eq:positive_association_for_values} continues to hold for all measurable, non-decreasing $\varphi,\psi:(\Z)^{\Z^d}\to[0,\infty)$.
     \end{lemma}

    \subsection{Disagreement percolation}\label{sec:disagreement_percolation}
    In this section we consider the operation of swapping finite clusters of disagreement of a pair of homomorphism height functions.

Say that a measure $\rho$ on the space\footnote{Functions in $(\Z^2)^{\Z^d}$ are written as $(f,g):\Z^d\to\Z^2$, where $f(v)$ is the first value at $v$ and $g(v)$ is the second value at $v$.} $(\Z^2)^{\Z^d}$ is a \emph{Gibbs measure for homomorphism pairs} if, when $(f,g)$ is sampled from $\rho$, then for any non-empty, finite $\Lambda\subset\Z^d$, the conditional distribution of $(f, g)_{\Lambda^+}$, conditioned on $(f, g)_{\Lambda^c}$, equals the product measure $\mu^{\Lambda,\tau_f}\otimes \mu^{\Lambda,\tau_g}$ with $\tau_f := f_{\extB\Lambda}$ and $\tau_g := g_{\extB\Lambda}$, almost surely (that is, $\rho$ conforms to the product specification). We note two simple properties which will be used repeatedly:
\begin{itemize}[leftmargin=20pt]
	\item If $\mu_1, \mu_2$ are Gibbs measures (for homomorphism height functions on $\Z^d$) then the product measure $\mu_1\otimes\mu_2$ is a Gibbs measure for homomorphism pairs.
	\item If $\rho$ is a Gibbs measure for homomorphism pairs and $(f,g)$ is sampled from $\rho$, then each of the marginal distributions of $f$ and of $g$ are Gibbs measures (for homomorphism height functions on~$\Z^d$).
\end{itemize}

   In the following lemma and later in the paper we shorthand expressions such as $\{v\in\Z^d\colon f_v>g_v\}$ to $\{f>g\}$. Write $\Cf$ for the (countable) collection of all finite, connected (non-empty) subsets of $\Z^d$.
    \begin{lemma}\label{lem:finite_cluster_swap}
      Let $(f,g)$ be sampled from a Gibbs measure for homomorphism pairs. Let $\eps:\Cf\to\{0,1\}$. Define a new pair of homomorphism height functions $(\tilde{f}, \tilde{g})$ as follows:
      \begin{equation}\label{eq:finite_cluster_swap_def}
        (\tilde{f},\tilde{g})(v) = \begin{cases}
          (g,f)(v)& \text{there is a finite connected component $C$ of $\{f\neq g\}$ containing $v$ and $\eps_C = 1$},\\
          (f,g)(v)& \text{otherwise}.
        \end{cases}
      \end{equation}
      Then $(\tilde{f}, \tilde{g})$ has the same distribution as $(f,g)$.
    \end{lemma}

    \begin{cor}\label{cor:finite_cluster_equalizing}
      Let $(f,g)$ be sampled from a Gibbs measure for homomorphism pairs. Let $\eps:\Cf\to\{0,1\}$. Define a new pair of homomorphism height functions $(\bar{f}, \bar{g})$ as follows:
      \begin{equation}\label{eq:finite_cluster_equalizing_def}
        (\bar{f},\bar{g})(v) = \begin{cases}
          (f,f)(v)& \!\!\!\text{there is a finite connected component $C$ of $\{f\neq g\}$ containing $v$, and $\eps_C = 0$},\\
          (g,g)(v)& \!\!\!\text{there is a finite connected component $C$ of $\{f\neq g\}$ containing $v$, and $\eps_C = 1$},\\
          (f,g)(v)& \!\!\!\text{otherwise}.
        \end{cases}
      \end{equation}
      Then $\bar{f}$ has the same distribution as $f$ and $\bar{g}$ has the same distribution as $g$.
    \end{cor}

    \begin{proof}[Proof of Lemma~\ref{lem:finite_cluster_swap}]
      One checks that if $\eps_n:\Cf\to\{0,1\}$ is a sequence converging to $\eps$ (i.e., for each $C\in \Cf$ there exists $n_C$ such that $(\eps_n)_C = \eps_C$ for all $n\ge n_C$) then $(\tilde{f}_n, \tilde{g}_n)$, defined via the recipe~\eqref{eq:finite_cluster_swap_def} with $\eps_n$ replacing $\eps$, converges (pointwise) to $(\tilde{f}, \tilde{g})$, almost surely. In particular, $(\tilde{f}_n, \tilde{g}_n)$ converges to $(\tilde{f}, \tilde{g})$ in distribution. Thus it suffices to prove the lemma under the assumption that $\eps_C=0$ for all but finitely many $C$; an assumption which we now impose.

      Fix a non-empty, finite $\Lambda\subset\Z^d$, large enough to satisfy that $C\subset\Lambda$ whenever $\eps_C = 1$. By definition, $(\tilde{f},\tilde{g})_{\Lambda^c} = (f,g)_{\Lambda^c}$. In addition, conditioned on $(f,g)_{\Lambda^c}$, the distribution of $(f,g)_{\Lambda}$ is uniform over all pairs of homomorphism height functions extending the boundary conditions, as $(f,g)$ is sampled from a Gibbs measure for homomorphism pairs. It thus suffices to prove that this latter uniformity statement holds also for $(\tilde{f}, \tilde{g})$. This now follows from the straightforward fact that, given $\eps$ and conditioned on $(f,g)_{\Lambda^c}$ (which equals $(\tilde{f},\tilde{g})_{\Lambda^c}$), the definition \eqref{eq:finite_cluster_swap_def} yields a bijection (in fact, an involution) between the homomorphism pairs $(f,g)_{\Lambda}$ and the homomorphism pairs $(\tilde{f}, \tilde{g})_{\Lambda}$ which extend the boundary conditions.
    \end{proof}
    \begin{proof}[Proof of Corollary~\ref{cor:finite_cluster_equalizing}]
      Lemma~\ref{lem:finite_cluster_swap} stipulates that the distribution of $(\tilde{f}, \tilde{g})$ (defined by~\eqref{eq:finite_cluster_swap_def}) equals the distribution of $(f, g)$, for any $\eps:\Cf\to\{0,1\}$. However, comparing the definitions~\eqref{eq:finite_cluster_swap_def} and \eqref{eq:finite_cluster_equalizing_def} makes it clear that $\bar{f} = \tilde{f}$ when the same sequence $\eps$ is used in both definitions. Thus $\bar{f}\eqd f$. Similarly, $\bar{g}=\tilde{g}$ when the sequence $\eps$ is used in~\eqref{eq:finite_cluster_equalizing_def} and the sequence $1-\eps$ is used in~\eqref{eq:finite_cluster_swap_def}. Thus $\bar{g} \eqd g$.
    \end{proof}

    \begin{cor}\label{cor:stochastic domination}
      Let $(f,g)$ be sampled from a Gibbs measure for homomorphism pairs. If
      \begin{equation}\label{eq:no_f_less_than_g_infinite_component}
        \P(\text{there is an infinite connected component of $\{f<g\}$}) = 0
      \end{equation}
      then the distribution of $f$ stochastically dominates the distribution of $g$. Similarly, if
      \begin{equation}\label{eq:no_f_different_from_g_infinite_component}
        \P(\text{there is an infinite connected component of $\{f\neq g\}$}) = 0
      \end{equation}
      then the distribution of $f$ equals the distribution of $g$.
    \end{cor}
    \begin{proof}
      Define $(\bar{f}, \bar{g})$ according to the recipe~\eqref{eq:finite_cluster_equalizing_def} with $\eps\equiv 0$. Assuming~\eqref{eq:no_f_less_than_g_infinite_component} we obtain that $\bar{f}\ge \bar{g}$, which yields the required stochastic domination as $\bar{f}\eqd f$ and $\bar{g}\eqd g$ by Corollary~\ref{cor:finite_cluster_equalizing}. Similarly, assuming~\eqref{eq:no_f_different_from_g_infinite_component} it holds that $\bar{f}=\bar{g}$ implying that $f\eqd g$.
    \end{proof}
    \begin{remark}
      The corollary generalizes, for homomorphism height functions, the disagreement percolation criterion of van den Berg \cite{MR1207673} who proved that the assumption~\eqref{eq:no_f_different_from_g_infinite_component} implies equality of the marginal distributions when the Gibbs measure from which $(f,g)$ are sampled is a product measure ($f$ is independent from $g$). Van den Berg-Maes \cite{van1994disagreement} also discuss generalizations but we have not seen Corollary~\ref{cor:stochastic domination} explicitly in the literature. We note that it is important that $(f,g)$ are sampled from a Gibbs measure for homomorphism pairs rather than just from a coupling of Gibbs measures with the same specification, as \cite{haggstrom2001note} gives an example of such a coupling of \emph{distinct} Gibbs measures for which~\eqref{eq:no_f_different_from_g_infinite_component} holds. We need this generalisation in the proof of Corollary \ref{corollary:mean}.

    \end{remark}

\begin{remark}
Cluster swapping was applied in \cite[Section 8.2]{MR2251117} to give short proofs of  log-concavity of the distribution of single-site marginals (See Section~\ref{section: Log-concavity}) and monotonicity of the Gibbs measures with respect to boundary conditions, for general class of random surfaces with convex interactions (see also \cite[Section 2]{cohen2017rarity}).
\end{remark}

  \begin{cor}\label{cor:coupling_to_get_extremality}
      Let $\mu_1, \mu_2$ be Gibbs measures and $f,g$ be \emph{independently} sampled from $\mu_1, \mu_2$, respectively. If~\eqref{eq:no_f_different_from_g_infinite_component} holds then the two measures are equal and \emph{extremal}.
    \end{cor}
\begin{proof}
The equality $\mu_1=\mu_2$ follows from Corollary~\ref{cor:stochastic domination}. We are left to prove that $\mu=\mu_1=\mu_2$ is extremal. Otherwise $\mu=\frac{1}{2}(\nu_1+\nu_2)$ for distinct Gibbs measures $\nu_1$ and $\nu_2$ (\cite[Theorem 14.15]{MR2807681}). Thereby the independent coupling of $\mu$ with itself would decompose as the average of the independent couplings of $\nu_1$ with itself, $\nu_2$ with itself, $\nu_1$ with $\nu_2$ and $\nu_2$ with $\nu_1$. As~\eqref{eq:no_f_different_from_g_infinite_component} holds for each of these couplings it follows from Corollary~\ref{cor:stochastic domination} that $\nu_1 = \nu_2$, a contradiction.
\end{proof}

    \subsection{Uniqueness of infinite clusters (following Burton-Keane)} It is well known in percolation theory on $\Z^d$ (and other amenable groups) that, almost surely, there is at most one infinite connected component of open sites. This fact was first proven by Aizenman--Kesten--Newman~\cite{aizenman1987uniqueness}, with a subsequent shorter proof provided by Burton--Keane~\cite{MR990777}. We show here that the ideas of Burton and Keane admit extension to the random set $\{f>g\}$ for suitable distributions of $(f,g)$. Similar arguments are used in~\cite[Lemma 8.5.6]{MR2251117}.
    \begin{thm}\label{thm:at_most_one_infinite_cluster}
      Let $(f,g)$ be sampled from an $\L$-translation-invariant Gibbs measure for homomorphism pairs, for a full-rank sublattice $\L$. Then
      \begin{equation*}
        \P(\text{the set $\{f>g\}$ has two or more infinite connected components}) = 0.
      \end{equation*}
    \end{thm}

    Fix an integer $M>0$. For a configuration $\omega\in\{0,1\}^{\Z^d}$, regarded as a subgraph of $\Z^d$, say that the graph ball $v + \Lambda(M)$ is a \emph{trifurcation ball} in $\omega$ if
    \begin{itemize}[leftmargin=20pt]
      \item there is an infinite connected component $\mathcal{C}$ of $\omega$ intersecting $v + \Lambda(M)$, and
      \item there is a connected subgraph $D\subset \mathcal{C}\cap(v + \Lambda(M))$ (induced from $\mathcal C$) such that $\mathcal{C}\setminus D$ has at least three infinite connected components.
    \end{itemize}
    The argument of Burton and Keane~\cite[Theorem 2]{MR990777} adapts to show that, for any $\omega\in\{0,1\}^{\Z^d}$, the density of trifurcation balls is zero. That is,
    \begin{equation*}
      \lim_{L\to\infty} \frac{1}{L^d}|\{v\in\Lambda(L)\colon \text{$v + \Lambda(M)$ is a trifurcation ball}\}| = 0.
    \end{equation*}
	For this equation to hold, it is important that $D\subset \mathcal{C}\cap(v + \Lambda(M))$ is a \emph{connected} subset in the definition of trifurcation balls. For instance, it does not hold if instead the definition of trifurcation balls only requires $\mathcal{C}\setminus (v + \Lambda(M))$, instead of $\mathcal{C}\setminus D$, to have at least three infinite connected components. Consider $\omega\in \{0,1\}^{\Z^2}$ such that $\omega_{v}=1$ if and only if the first coordinate of $v$ is even or the second coordinate is zero. For such $\omega$ we would have (by the alternative definition) that $(v + \Lambda(3))$ is a trifurcation ball for all $v\in \Z^d$.

    The following statement is an immediate consequence, making use of the ergodic theorem.
    \begin{lemma}\label{Lemma: no trifurcation}
      Let $\P$ be a measure on $\{0,1\}^{\Z^d}$ which is $\L$-translation-invariant for some full-rank sublattice $\L$. Then
      \begin{equation*}
        \P(\Lambda(M)\text{ is a trifurcation ball}) = 0.
      \end{equation*}
    \end{lemma}
    The proof of Theorem \ref{thm:at_most_one_infinite_cluster} relies on Lemma~\ref{Lemma: no trifurcation} and Lemma~\ref{Lemma:flatness_gives_extreme_heights} given below.

    {\bf Flatness of $\L$-translation-invariant measures:} Fix a full rank sublattice $\L\subset \Z^d$. Let us recall some basic results about $\L$-translation-invariant probability measures on homomorphism height functions.

    Suppose $f$ is a given homomorphism height function and $\Delta\subset \Z^d$. Consider the homomorphism height functions given by
$$f_{\mathrm{max}, \Delta}(v)=\min\{f(w)+\|v-w\|_1\colon w\in \Delta\},$$
$$f_{\mathrm{min}, \Delta}(v)=\max\{f(w)-\|v-w\|_1\colon w\in \Delta\}.$$
It follows that
$$(f_{\mathrm{max}, \Delta})_\Delta=(f_{\mathrm{min}, \Delta})_\Delta=f_{\Delta}$$
and if $g$ is a homomorphism height function such that $f_\Delta=g_\Delta$ then pointwise
$f_{\mathrm{min}, \Delta}\leq g\leq f_{\mathrm{max}, \Delta}$.

    \begin{lemma}\label{Lemma:flatness_gives_extreme_heights} Let $f$ be sampled from an $\L$-translation-invariant probability measure on homomorphism height functions. For all positive integers $M$ there exists $m$ such that
\begin{equation}\label{equation: minimal maximal height}
\begin{split}
\P((f_{\mathrm{max}, \extB \Lambda(m)})_{\Lambda(M)}> \frac{m}{2})&> \frac{3}{4}\\
\P((f_{\mathrm{min}, \extB \Lambda(m)})_{\Lambda(M)}< -\frac{m}{2})&> \frac{3}{4}
\end{split}.
\end{equation}
    \end{lemma}
In the following proof we will use a version of the Poincar\'e recurrence theorem: Given a stationary process $(X_n)$ and an event $A$, if $\P((X_n)\in A)>0$ then $\P((X_{n+k})\in A\text{ for infinitely many $k>0$}\,|\,(X_n)\in A)=1$.
  \begin{proof}
  The following limit exists almost surely, for each $v\in\Z^d$, by the ergodic theorem
    \begin{equation}\label{eq:ergodic_limit}
      \lim_{n\to \infty}\frac{f(n\, v)-f(\zero)}{n}=\lim_{n\to \infty}\frac{1}{n}\sum_{i=0}^{n-1}\left(f((i+1)\, v)-f(i\, v)\right).
    \end{equation}
    For the stationary process $\{f(n\, v)\colon n \in \Z\}$, if we have for some $k \in \Z$ the occurence of a positive probability event $f(\zero)=k$, then by the Poincar\'e recurrence theorem $f(n\, v)=k$ for infinitely many $n>0$ almost surely, conditioned on $f(\zero)=k$. As a consequence it follows that $f(n\, v)=f(\zero)$ for infinitely many $n$ and the limit in \eqref{eq:ergodic_limit} is necessarily $0$ almost surely.
These statements can be used to prove that, almost surely,
\begin{equation}
\lim_{n\to \infty}\sup_{|v|_1=n}\frac{|f(v)|}{n}=0.\label{equation: slope of ergodic measure}
\end{equation}
A detailed proof can be found in~\cite[Lemma 6.4]{MR3552299}. Its idea is as follows: the previous statements on the limit~\eqref{eq:ergodic_limit} extend to
\begin{equation}\label{eq:ergodic_limit2}
  \lim_{n\to \infty}\frac{f(\lfloor n\, v\rfloor)}{n} = 0,\quad v\in\R^d,
\end{equation}
where $\lfloor\cdot\rfloor$ indicates rounding down each coordinate to the next integer. For $v\in\Z^d$ this follows from the previous statements, and this case implies also the case that $v$ has rational coordinates, making use of the homomorphism property (the fact that $f$ changes by $1$ between adjacent vertices). A density argument and another use of the homomorphism property then yields the full statement~\eqref{eq:ergodic_limit2}. Now apply~\eqref{eq:ergodic_limit2} for a finite $\eps$-dense set in the sphere $\{v\in\R^d\colon\|v\|_1=1\}$ and use the homomorphism property to conclude that, almost surely, $\sup_{|v|_1=n}\frac{|f(v)|}{n}\le 2\eps$ for all $n$ larger than some random threshold. This implies~\eqref{equation: slope of ergodic measure} as $\eps$ is arbitary.

By \eqref{equation: slope of ergodic measure}, for all $\epsilon, \delta>0$, there exists a positive integer $M$ such that for all $m>M$,
\begin{equation}
\P(f_{\extB \Lambda(m)}\in (-m \epsilon, m \epsilon))>1-\delta.\label{equation: height less}
\end{equation}
From this \eqref{equation: minimal maximal height} follows.
  \end{proof}


In the following proof we will need some basic facts about the ergodic decomposition. The ergodic decomposition gives us that any  $\L$-translation-invariant measure is a mixture of $\L$-ergodic measures. By a standard result \cite[Theorem 14.15]{MR2807681},  the ergodic decomposition of an $\L$-translation-invariant Gibbs measure consists of $\L$-ergodic Gibbs measures.

\begin{proof}[Proof of Theorem \ref{thm:at_most_one_infinite_cluster}]
As stated above, by \cite[Theorem 14.15]{MR2807681}, it is sufficient to prove the result for $\L$-ergodic Gibbs measures for homomorphism pairs.

Suppose that there are infinitely many infinite connected components of $\{f>g\}$. Choose $M$ large enough such that
$$\P(\text{at least three infinite connected components of ${f>g}$ intersect $\Lambda(M)$})>\frac{3}{4}.$$
By \eqref{equation: minimal maximal height} choose $m>M$ such that
\begin{eqnarray*}
\P((f_{\mathrm{max}, \extB \Lambda(m)})_{\Lambda(M)}> \frac{m}{2})&>& \frac{3}{4}\\
\P((g_{\mathrm{min}, \extB \Lambda(m)})_{\Lambda(M)}< -\frac{m}{2})&>& \frac{3}{4}.
\end{eqnarray*}
Putting these together we have that
\begin{eqnarray*}
\P(\text{at least three infinite connected components of ${f>g}$ intersect $\Lambda(M)$},&&\\
(f_{\mathrm{max}, \extB \Lambda(m)})_{\Lambda(M)}> \frac{m}{2})\text{ and } ((g_{\mathrm{min}, \extB \Lambda(m)})_{\Lambda(M)}< -\frac{m}{2}))&>& \frac14.
\end{eqnarray*}
Suppose $(f, g)$ is a homomorphism pair such that at least three infinite connected components $\mathcal C_1, \mathcal C_2, \mathcal C_3$ of ${f>g}$ intersect $\Lambda(M)$,
$((f_{\mathrm{max}, \extB \Lambda(m)})_{\Lambda(M)}> \frac{m}{2})$ and  $((g_{\mathrm{min}, \extB \Lambda(m)})_{\Lambda(M)}< -\frac{m}{2}).$ Then there exists a homomorphism pair $(\tilde f, \tilde g)$ such that $\tilde f>f$, $\tilde g<g$,
$$(\tilde f, \tilde g)_{\Z^d\setminus \Lambda(m)}= ( f, g)_{\Z^d\setminus \Lambda(m)}$$
and $\mathcal C_1\cup \mathcal C_2 \cup \mathcal C_3\cup \Lambda(M)$ is contained in an infinite connected component of $\tilde f> \tilde g$. In this case, $\Lambda(m)$ is a trifurcation ball for $\tilde f> \tilde g$.

Since $(f,g)$ are sampled from a Gibbs measure for homomorphism pairs it follows that
\begin{equation*}
\P(\Lambda(m) \text{ is a trifurcation ball})>0.
\end{equation*}
By Lemma~\ref{Lemma: no trifurcation}, we have a contradiction.

Now suppose that $\{f>g\}$ has more than one (but finitely many) connected components. Choose $M$ such that
$$\P(\text{all the infinite connected components of $\{f>g\}$ intersect $\Lambda(M)$})>\frac{3}{4}.$$
By the argument above we get that
$$\P(\{f>g\}\text{ has a unique connected component})>0.$$
But the number of connected components is constant almost surely. Hence there is at most one infinite connected component almost surely.
\end{proof}


%

    \subsection{No coexistence of primal and dual percolation with unique infinite cluster}\label{sec:no_coexistence}

    In the classical theory of percolation, it was shown by Harris \cite{MR0115221} that in bond percolation on $\Z^2$ there is no infinite cluster at $p=1/2$, almost surely. Indeed, if the infinite cluster had positive probability to appear, then it would appear almost surely, and by symmetry, also the dual percolation would have an infinite cluster almost surely. However, in a sense, planar geometry places too many constraints on how the two clusters should co-exist for this to be possible. This argument has been generalized considerably in order to weaken the required assumptions. The version required here is based on an idea of Zhang.

    \begin{thm}\cite[Theorem 14.3]{MR2280297}\label{thm: no coexistence}
   For a full-rank sublattice $\L$, if $\mu$ is a probability measure on $\{0,1\}^{\Z^2}$ satisfying
    \begin{enumerate}
    \item
(Ergodicity) $\mu$ is $\L$-ergodic.
\item
(Positive association) For any two increasing events $A,B$, $\P(A\cap B)\ge \P(A)\P(B)$.
\item (Symmetry) Invariance under interchange of coordinates and reflection in coordinate hyperplanes.
\item (Unique infinite cluster) There is at most one infinite connected component of $0$'s and at most one infinite connected component of $1$'s.
\end{enumerate}
Then the probability that there is simultaneous existence of an infinite connected component of $0$'s and an infinite connected component of $1$'s is zero.
    \end{thm}

Strictly speaking, \cite[Theorem 14.3]{MR2280297} is for measures satisfying assumption (1) with $\L=\Z^2$, (2), (3) and an additional ``finite energy'' assumption which means roughly that every pattern appears with positive probability. However the proof given there carries forward to $\L$-ergodic probability measures as well and the only place where the ``finite energy'' assumption is used is to prove assumption (4) above.

A further generalization which removes the symmetry assumption is proved in (\cite[Theorem 9.3.1 and Corollary 9.4.6]{{MR2251117}}) and also in \cite[Theorem 1.5]{DCRT2017}.

\subsection{$\mathcal L$-invariant Gibbs measures are preserved under swapping of infinite clusters}\label{subsection: infinite clusters}
In Section \ref{sec:disagreement_percolation} we showed that the operation of swapping finite clusters of disagreement of pairs of homomorphism height functions sampled from a Gibbs measure preserves the Gibbs property. In this section we will prove the same holds when we swap infinite clusters instead. 

   \begin{lemma}\label{lem:infinite_cluster_swap}
   Let $f,g$ be independent samples from two $\L$-ergodic Gibbs measures. Denote
	\begin{equation}\label{eq:C_f_greater_g_def}
	C = \{v\colon \text{$v$ belongs to an infinite connected component of $\{f>g\}$}\}.
	\end{equation}
	Define a new pair $(\tilde{f}, \tilde{g})$ of homomorphism height functions on $\Z^d$ by
	\begin{equation}\label{eq:swapping_on_C}
	(\tilde{f}(v), \tilde{g}(v)) := \begin{cases}
	(f(v), g(v))&v\notin C,\\
	(g(v), f(v))&v\in C.
	\end{cases}
	\end{equation}
	Then the distribution of $(\tilde{f}, \tilde{g})$ is a $\mathcal L$-invariant Gibbs measure for homomorphism pairs.
\end{lemma}

 Unlike with finite cluster swaps, this lemma is no longer true without the assumption that the Gibbs measure is $\mathcal L$-invariant: Denote the coordinates of $v\in \Z^d$ by $(v_1, v_2,\ldots,v_d)$.
Consider homomorphism pairs $(f, g)$ given by 
$$f(v):=v_1+v_2+\ldots +v_d\text{ and }g(v):=-f(v)\text{ for all }v\in \Z^d.$$ 
The delta mass on $(f(v), g(v))$ forms a Gibbs measure because these configurations are ``frozen'' meaning that if $(f', g')$ is a homomorphism pair which agrees with $(f, g)$ on all but finitely many sites then it must be equal to $(f, g)$. On the other hand, after applying the cluster swaps we get the delta mass on the pair $(\tilde f, \tilde g)$ given by 
$$\tilde f(v):=|v_1+v_2+\ldots +v_d|\text{ and }\tilde g(v):=-|v_1+v_2+\ldots +v_d|\text{ for all }v\in \Z^d$$ 
which is clearly not a Gibbs measure because $\mu^{\Lambda(0), \tilde f_{\partial_0 \Lambda(0)}}$ has the uniform distribution on values $\{0,2\}$ at the origin.

The proof of this lemma will go via a variational principle where $\mathcal L$-invariant Gibbs measures are identified precisely as the measures of maximal entropy for this model.

For this we begin by defining entropy as is required in our case. Given a random variable $X$ taking values in a finite set $X$, the Shannon entropy of $X$ is defined as 
$$H(X):=\sum_{a\in A}-\mathbb P (X=a)\log(\mathbb P(x=a)).$$
We record two basic facts about Shannon entropy. The Shannon entropy for random variables taking values in a fixed finite state space $S$ is maximised by the uniform distribution on $S$. This follows from a straightforward application of the Jensen's inequality. Also if $X$ and $Y$ are independent, a straightforward computation shows that
$$H(X, Y)=H(X)+H(Y).$$
The entropy of a random field is given by the growth rate of Shannon entropy on the field restricted to finite regions. For this, given a function $f: \Z^d\to S$ where $S$ is $\Z$ or $\Z^2$ we write $(f\wedge M)$ to denote the truncated height function 
$$(f\wedge M)(v):=\begin{cases}f(v)&\text{ if  }\|f(v)\|_1<M\\
M&\text{otherwise}\end{cases}$$
where $\wedge$ denotes the minimum. Now if $f$ is a sample from a $\mathcal L $-invariant random field we define 
$$h(f):=\lim_{M\to \infty}\lim_{N\to \infty}\frac{1}{\Lambda(N)}\lim H(f\wedge M)_{\Lambda(N)}.$$
The limit with respect to $N$ exists by subadditivity arguments and the limit with respect to $M$ exists because $H(f\wedge M)_{\Lambda(N)}$ is monotonically increasing in $M$. Details can be found in \cite[Chapter 3]{keller1998equilibrium}. 
It follows automatically that if $(f, g)$ are independent $\mathcal L$-invariant random fields then
\begin{equation}
h(f, g)=h(f)+h(g).\label{equation:entropy add}
\end{equation}

We recall now that entropy is preserved under countable-to-one maps.
\begin{prop}\label{prop: countable to one maps}
	Let $f$ be sampled from a $\mathcal L$-invariant random field $\mu$ with state space $S$ and $\phi$ be an equivariant map on the support of $\mu$ which is countable to one $\mu$ almost everywhere. Then $h(f)=h(\phi(f))$.
	\end{prop}
The proof of this classical fact can be found in \cite[Theorem 4.1.15]{downarowicz2011entropy} for $\Z$-indexed invariant random fields but can be easily adapted to our setting with standard changes. However wherever we will use this proposition we will sketch some alternative arguments. This proposition implies that the entropy of $\mathcal L$-invariant measures on homomorphism height functions is bounded. Indeed the map $f\to f\!\!\!\mod3$ which maps homomorphism height functions surjectively to proper $3$-colorings of $\Z^d$ is countable to one: Given $f\!\!\!\mod3$, $f$ is completely determined up to a choice of $f(0)$. Hence the entropy of any $\mathcal L$-invariant measure on homomorphism height functions is bounded by the maximal entropy of $\mathcal L $-invariant measures on the space of proper $3$-colorings. Here is an alternative way to prove that $h(f)=h(f\!\!\!\mod 3)$: It is easy to see that
$h(f\!\!\!\mod 3)\leq h(f)$. For proving the opposite inequality notice that $f_{\Lambda(n)}$ is completely determined by $f(0)$ and $(f\!\!\!\mod 3)_{\Lambda_n}$. 
Thus
$$H((f\wedge M)_{\Lambda(n)})\leq H((f\!\!\!\mod 3)_{\Lambda(n)})+H(f(0)\wedge (M+n))\leq H((f\!\!\!\mod 3)_{\Lambda(n)}) +\log(2(M+n)).$$

From here it follows by dividing by $|\Lambda(n)|$ and taking the respective limits shows that $h(f\wedge M)\leq h(f\!\!\!\mod 3)$. Finally by taking $M\to \infty$ we get that $h(f)\leq h(f\!\!\!\mod 3)$ This completes the proof.

We will need another important property of these maps. If $f$ is sampled from a Gibbs measure on homomorphism height functions then $f\!\!\!\mod3$ is sampled from Gibbs measure on proper $3$-colorings. This follows because given any finite connected set $\Lambda$ with a connected complement the map $\!\!\!\mod3$  is surjective and bijective from the support of $\mu^{\Lambda, {f}_{\extB \Lambda}}$ on homomorphism height functions to $\mu^{\Lambda, {(f\!\!\!\mod 3)}_{\extB \Lambda}}$ on proper $3$ colorings.

We begin by proving that the Gibbs measures on homomorphism height functions have maximal entropy.
\begin{lemma} \label{lemma: Gibbs has maximal entropy}
$\mathcal L$-invariant Gibbs measures have the maximal entropy among all $\mathcal L$-invariant measures on homomorphism height functions.
	\end{lemma}

\begin{proof}
Fix $\epsilon>0$. Let $f$ be sampled from an $\mathcal L$-invariant measure on the homomorphism height functions. By \eqref{equation: height less}, there exists a positive integer $M$ such that for all $m>M$,
\begin{equation}
\P(f_{\extB \Lambda(m)}\in (-m \epsilon, m \epsilon))>1-\epsilon.
\end{equation}
Let $f$ be a homomorphism height function such that $f_{\extB \Lambda(m)}\in (-m \epsilon, m \epsilon)$. Let $n$ be the largest even integer such that $n<m(1-\epsilon)$ and $g$ be a homomorphism height function such that 
$$g(v)=0 \text{ for all }\|v\|_1=n.$$ 
Then we have that $|f(w)-g(w')|<\|w-w'\|_1$ for all $w\in \Z^d\setminus\Lambda(m)$ and $w'\in \Lambda(n)$.
It follows that there exists a homomorphism height function $k$ such that 
$$k_{\Lambda(n)}= f_{\Lambda(n)}\text{ and }k_{\Z^d\setminus \Lambda(m)}=f_{\Z^d\setminus \Lambda(m)}.$$
For this consider the function
$$k(v):=\min_{w\in \Z^d\setminus \Lambda(m),w'\in \Lambda(n)}\{f(w)+\|w-v\|_1, g(w')+\|w'-v\|_1\} \text{ for }v\in \Z^d.$$
It is easy to check that this defines the required homomorphism height function. 

This implies that 
$$H(f_{\Lambda(m)}\!\!\!\mod 3)\geq (1-\epsilon) H(\mu^{\Lambda(n-1), 0}).$$
Let $h$ be the maximal entropy achieved by $\mathcal L$-invariant measures on proper $3$-colorings. By \cite[Proposition 2.1]{benjamini2000random} it follows that
$$\lim_{n\to \infty}\frac{1}{|\Lambda(n)|}H(\mu^{\Lambda(n-1), 0})=h$$
Thus 
$$h(f)=h(f\!\!\!\mod3)>\lim_{m\to \infty}\frac{|\Lambda(\lfloor(1-\epsilon)m\rfloor-2)|}{|\Lambda(m)|}(1-\epsilon)h> (1-2\epsilon)^{d+1}h.$$ 

 Since $\epsilon$ was arbitrary the proof is complete.
	\end{proof}
\begin{remark}
	The fact that  the common extension of $f_{\Z^d\setminus \Lambda(m)}$ and $g_{\Lambda(n)}$ to a homomorphism height function exists follows more directly from the Kirszbraun theorem \cite{chandgotia2017kirszbraun,MT2017}.
\end{remark}

\begin{lemma}\label{lemma: cluster swap no reduce}
	Let $(f, g)$ and $(\tilde f, \tilde g)$ be as in Lemma \ref{lem:infinite_cluster_swap} then 
	$h(f, g)=h(\tilde f, \tilde g)$.
	\end{lemma}
\begin{proof}
	By Theorem \ref{thm:at_most_one_infinite_cluster}, with probability one both $\{f>g\}$ and $\{f<g\}$ have at most one infinite connected component each. Thus the map $(f, g)\to (\tilde f, \tilde g)$ is at most $2$ to $1$. By Proposition \ref{prop: countable to one maps}, the entropy must be preserved.
	\end{proof}
\begin{remark}
There is a more direct proof which does use neither $\{f>g\}$ and $\{f<g\}$ have at most one infinite connected component each nor Proposition \ref{prop: countable to one maps}. For this just use the fact that $(f,g)_{\Lambda(n)}$ is a function of $(f, g)_{\partial_0\Lambda(n)}$ and $(\tilde f, \tilde g)_{\Lambda(n)}$ and follow it with some simple direct computations.
\end{remark}
The final component of the proof that we need is the Lanford-Ruelle theorem \cite{lanford1969observables}. We do not need the fact at this level of generality and a much simpler proof by Burton and Steif suffices for our setting \cite[Proposition 1.19]{burton1994non}.
\begin{lemma}\label{lemma:maximal entropy is Gibbs}
	If a $\mathcal L$-invariant measure on homomorphism pairs has maximal entropy then it is a Gibbs measure.
	\end{lemma}
While the proofs in \cite{lanford1969observables,burton1994non} are for a finite state space it also extends to our setting since the entropy of measures on homomorphism pairs is bounded. Here is a sketch of the proof in this case. Suppose $\mu$ is a $\mathcal L$-invariant measure on homomorphism pairs. If it is not a Gibbs measures, there exists $m$ such that $(f, g)_{\Lambda(m)}$ conditioned on $(f, g)_{\partial_0\Lambda(m)}$ is not uniformly distributed. However since the uniform measure is the unique maximiser of Shannon entropy, the entropy of $(f, g)_{\Lambda(m)}$ conditioned on $(f, g)_{\partial_0\Lambda(m)}$ increases when it is replaced by a sample from $\mu^{\Lambda(m), (f, g)_{\partial_0\Lambda(m)}}$. By the ergodic theorem, we can find disjoint appearances of the pattern $(f, g)_{\partial_0\Lambda(m)}$ on a positive density of positions in $\Z^d$. Since the entropy can be improved on all of these positions by the same amount, we get an average increase in entropy per site. Hence $\mu$ could not have been the measure of maximal entropy to start with.

\begin{proof}[Proof of Lemma \ref{lem:infinite_cluster_swap}]
We know from Lemma \ref{lemma: Gibbs has maximal entropy} that $f$ and $g$ have maximal entropy among $\mathcal L$-invariant measures on homomorphism height functions. Since $f$ and $g$ are independent we have from \eqref{equation:entropy add} that $(f,g)$ has maximal entropy among $\mathcal L$-invariant measures on homomorphism pairs. By Lemma \ref{lemma: cluster swap no reduce} we have that $(\tilde f,\tilde g)$ has maximal entropy among $\mathcal L$-invariant measures on homomorphism pairs. Finally from Lemma \ref{lemma:maximal entropy is Gibbs} it follows that $(\tilde f,\tilde g)$ is also an $\mathcal L$-invariant Gibbs measure.
\end{proof}

    \section{Ergodic Gibbs measures are extremal}\label{section:ergodic_Gibbs_extremal}
The main goal of this section is to prove Theorem~\ref{thm:ergodic_Gibbs_measures_are_extremal} and its analogue for $d=2$.

Fix $d\ge 2$. Suppose $\mu$ is an $\L$-ergodic Gibbs measure for a full-rank sublattice $\L$. We  will first show that $\mu$ is extremal. Let $f,g$ be independently sampled from two $\L$-ergodic Gibbs measures $\mu_1$ and $\mu_2$. Define the events
    \begin{align*}
      E^+&:=\{\text{the set $\{f>g\}$ contains an infinite connected component}\},\\
      E^-&:=\{\text{the set $\{f<g\}$ contains an infinite connected component}\}.
    \end{align*}
\begin{lemma}\label{lem:disjointness of oppositve events for f>g and f<g}
If $f,g$ are independent samples from an $\L$-ergodic Gibbs measure on homomorphism height functions then
\begin{equation}\label{eq:no_infinite_clusters_for_extremality}
\P(E^+) = \P(E^-) = 0.
\end{equation}
\end{lemma}

By Corollary~\ref{cor:coupling_to_get_extremality} this will imply that $\mu$ is extremal. In the following lemma we show that $E^+$ and $E^-$ cannot occur simultaneously.
    \begin{lemma}\label{lem:no_infinite_clusters_for_f>g_and_f<g} If $f,g$ are independent samples from two $\L$-ergodic Gibbs measures $\mu_1$ and $\mu_2$ respectively then
      \begin{equation}\label{eq:no_infinite_clusters_for_f>g_and_f<g}
        \P(E^+\cap E^-) = 0.
      \end{equation}
    \end{lemma}
    \begin{proof}
      Define a new pair of functions $\tilde{f}, \tilde{g}$ according to the recipe \eqref{eq:C_f_greater_g_def}, \eqref{eq:swapping_on_C} given in Lemma~\ref{lem:infinite_cluster_swap}. The construction makes it clear that the occurrence of $E^+\cap E^-$ implies the occurrence of the event
      \begin{equation*}
        E = \{\text{the set $\{\tilde{f}<\tilde{g}\}$ has two or more infinite connected components}\}
      \end{equation*}
      Thus it suffices to show that $\P(E)=0$. To this end, apply Lemma~\ref{lem:infinite_cluster_swap} to conclude that the distribution of $(\tilde{f}, \tilde{g})$ is a Gibbs measure for the homomorphism pairs. In addition, the construction implies that this Gibbs measure retains the $\L$-translation-invariance property of $(f,g)$. Thus we may apply Theorem~\ref{thm:at_most_one_infinite_cluster} (to the pair $(\tilde{g},\tilde{f})$) to conclude that $\P(E)=0$, as required.
    \end{proof}
    In the proof of Theorem~\ref{thm:ergodic_Gibbs_measures_are_extremal} we will use the fact that $\mu$ is $\L$-ergodic for a full-rank sublattice $\L$ in the following way: it would imply that there exists a sequence $(p_k)_{k\in\Z}$ so that
    \begin{equation}\label{eq:value_statistics_ergodicity}
      \lim_{L\to\infty} \frac{1}{|\Lambda(L)|}|\{v\in\Lambda(L)\colon f(v) = k\}| = p_k\quad\text{for each $k\in\Z$, almost surely}.
    \end{equation}
    Let us also define a function $I(f,g):\Z^d\to\{0,1\}$ by
    \begin{equation*}
      I(f,g)(v) = 1\text{ if and only if $v$ belongs to an infinite connected component of $\{f>g\}$}.
    \end{equation*}
\begin{proof}[Proof of Lemma \ref{lem:disjointness of oppositve events for f>g and f<g}]
The $\L$-translation-invariance of the product measure of $\mu$ with itself implies that
\begin{equation}\label{eq:positive_density_of_cluster}
\lim_{L\to\infty} \frac{1}{|\Lambda(L)|}\sum_{v\in\Lambda(L)}I(f,g)(v)>0,\quad\text{almost surely on the event $E^+$}.
\end{equation}
Now define a pair of homomorphism height functions $(\bar{f}, \bar{g})$ via the recipe~\eqref{eq:finite_cluster_equalizing_def} applied with $\eps\equiv 0$. Then
\begin{equation}\label{eq:bar_f_bar_g_properties}
\begin{split}
&\bar{f} \ge \bar{g}\quad\text{on the event $(E^-)^c$},\\
&I(f,g) = I(\bar{f},\bar{g}).
\end{split}
\end{equation}
In addition, Corollary~\ref{cor:finite_cluster_equalizing} shows that $\bar{f}\eqd f$ and $\bar{g} \eqd g$. In particular, referring to \eqref{eq:value_statistics_ergodicity},
\begin{equation}\label{eq:value_statistics_bar_f_bar_g}
\lim_{L\to\infty} \frac{1}{|\Lambda(L)|}|\{v\in\Lambda(L)\colon \bar{f}(v) = k\}| = \lim_{L\to\infty} \frac{1}{|\Lambda(L)|}|\{v\in\Lambda(L)\colon \bar{g}(v) = k\}|\quad\text{for each $k\in\Z$, almost surely}.
\end{equation}
Assume now, in order to obtain a contradiction, that (by symmetry) $\P(E^+) = \P(E^-) > 0$. Lemma~\ref{lem:no_infinite_clusters_for_f>g_and_f<g} then implies that $\P(E^+\setminus E^-)>0$. Thus, combining \eqref{eq:positive_density_of_cluster} with \eqref{eq:bar_f_bar_g_properties} we conclude that
\begin{equation*}
\bar{f} \ge \bar{g}\text{ and }\lim_{L\to\infty} \frac{1}{|\Lambda(L)|}|\{v\in \Lambda(L)\colon \bar{f}(v)>\bar{g}(v)\}|>0,\quad\text{almost surely on the event $E^+\setminus E^-$},
\end{equation*}
in contradiction with \eqref{eq:value_statistics_bar_f_bar_g}.
\end{proof}
  We conclude that~\eqref{eq:no_infinite_clusters_for_extremality} holds and by Corollary~\ref{cor:coupling_to_get_extremality} we have that $\mu$ is extremal.

    \begin{cor}\label{corollary:stochastic_domination}
Let $\mu_1$ and $\mu_2$ be $\L$-ergodic Gibbs measures for a full-rank sublattice $\L$. Either $\mu_1$ stochastically dominates $\mu_2$ or vice versa.
    \end{cor}

 \begin{proof}
 Let $(f,g)$ be sampled from the independent product of $\mu_1$ and $\mu_2$. By extremality of $\mu_1$ and $\mu_2$ we have that the probability of $\{f>g\}$ having an infinite connected component is either $0$ or $1$.  If it is $1$ then by Lemma~\ref{lem:no_infinite_clusters_for_f>g_and_f<g} we have that, almost surely, all components of $\{f < g\}$ are finite. Thus the corollary follows from Corollary~\ref{cor:stochastic domination}.
 \end{proof}

We are left to prove that if $\mu$ is an $\L$-ergodic Gibbs measures then it is in fact $\Ze^d$-invariant.

 \begin{cor}\label{cor: invariance under permutation of coordinates and reflections}
For a sublattice $\mathcal{L}$ of full rank, the $\L$-ergodic Gibbs measures are invariant under translations by $\Ze^d$, permutation of coordinates and reflections in coordinate hyperplanes.
\end{cor}
  This will enable us to rely on Theorem \ref{thm: no coexistence} in the later parts of the paper.
\begin{proof}
Let $S$ be one of the following maps on the space of probability measures on $\Z^{\Z^d}$: translation by a vector in $\Ze^d$, reflection along coordinate hyperplanes or permutation of coordinates.
Given an $\L$-ergodic Gibbs measure $\mu$, $S(\mu)$ is also an $\L$-ergodic Gibbs measure. If $\mu$ and $S(\mu)$ are distinct, by Corollary~\ref{corollary:stochastic_domination} we have that either $\mu$ strictly stochastically dominates $S(\mu)$ or vice versa. Without loss of generality assume that $\mu$ strictly stochastically dominates $S(\mu)$. If $(f,g)$ is a sample from a coupling of $\mu$ and $S(\mu)$ such that $f\geq g$ then $S^n(f)\geq S^{n}(g)$ for all positive integers $n$. Since $(S^n(f),S^n(g))$ is a sample from a coupling of $S^n(\mu)$ and $S^{n+1}(\mu)$, it follows that $S^n(\mu)$ strictly stochastically dominates $S^{n+1}(\mu)$. Thus for all positive integers $n$, $\mu$ strictly stochastically dominates $S^n(\mu)$. Since there exists a positive integer $n$ such that $S^n(\mu)=\mu$ this leads to a contradiction and completes the proof.
\end{proof}

    \section{Uniqueness up to additive constant of ergodic Gibbs measures}
    In this section we prove Theorem~\ref{thm:uniqueness_up_to_additive_constant}.

    Fix the dimension $d=2$ throughout the section. Let $f,g$ be independent samples from $\Ze^2$-ergodic Gibbs measures $\mu, \mu'$. Our goal is to show that there is an integer $k_0$ such that the distribution of $f$ equals the distribution of $g + 2k_0$.

    Applying Theorem~\ref{thm:ergodic_Gibbs_measures_are_extremal} we have that $\mu$ and $\mu'$ are extremal Gibbs measures. A main consequence of extremality that we shall use is that
    \begin{equation*}
      \text{both $f$ and $g$ have positive association (for the pointwise partial order)},
    \end{equation*}
    according to Lemma~\ref{lem:FKG_for_pointwise_partial_order}. Since $(f,-g)$ have been sampled independently and the random fields of the form $1_{f>g+2k}$ are increasing functions of $(f, -g)$, they have positive association as well.

    For each integer $k$, define the events
    \begin{align*}
      E^+_k&:=\{\text{the set $\{f\ge g+2k\}$ contains an infinite connected component}\},\\
      E^-_k&:=\{\text{the set $\{f<g+2k\}$ contains an infinite connected component}\}.
    \end{align*}
    As another consequence of extremality, we note that the (joint) distribution of $(f,g)$ is an extremal $\Ze^2$-translation-invariant Gibbs measure for homomorphism pairs. Thus, for each integer $k$,
    \begin{equation}\label{eq:E_+_E_-_zero_one}
      \P(E^+_k),\, \P(E^-_k)\in \{0,1\}.
    \end{equation}
    The following statement is the main claim that we use to prove Theorem~\ref{thm:uniqueness_up_to_additive_constant}.
    \begin{lemma}\label{lem:no_coexistence_for_mu_and_mu'}
      For each integer $k$, $\P(E^+_k\cap E^-_k) = 0$.
    \end{lemma}
    \begin{proof}
Suppose for contradiction that $\P(E^+_k\cap E^-_k)=1$. By Theorem \ref{thm:at_most_one_infinite_cluster}, $\{f\geq g+2k\}$ and $\{f< g+2k\}$ each have a unique infinite connected component. In addition, we have by Corollary~\ref{cor: invariance under permutation of coordinates and reflections} and the discussion above	that
\begin{enumerate}
\item
$1_{f\geq g+2k}$ has positive association.
\item
$1_{f\geq g+2k}$ is invariant  under interchange of coordinates and reflection in coordinate hyperplanes.
\end{enumerate}
This contradicts Theorem \ref{thm: no coexistence}.
    \end{proof}
    We continue to study the events $E^+_k$ and $E^-_k$. It is clear that
    \begin{equation}\label{eq:E_+_E_-_monotonocity}
      \P(E^+_{k+1})\le \P(E^+_k),\quad \P(E^-_{k+1})\ge \P(E^-_k).
    \end{equation}
    Combining the relations~\eqref{eq:E_+_E_-_zero_one} and~\eqref{eq:E_+_E_-_monotonocity} with Lemma~\ref{lem:no_coexistence_for_mu_and_mu'}, and exchanging the roles of $f$ and $g$ if necessary (which switches $E_{k}^-$ with $E_k^+$), we see that one of the following cases must occur:
    \begin{enumerate}
      \item There exists an integer $k_0$ such that $\P(E^+_{k_0+1})=0$ and $\P(E^-_{k_0}) = 0$.
      \item For all integer $k$, $\P(E^-_k) = 0$.
    \end{enumerate}
    Note, however, that if $\P(E^-_k) = 0$ for some $k$ then the distribution of $f$ stochastically dominates the distribution of $g+2k$ by Corollary~\ref{cor:stochastic domination}. Thus the second case implies that $f$ dominates $g+2k$ for all integer $k$. This cannot occur. Suppose then that the first case occurs for some integer $k_0$. Applying Corollary~\ref{cor:stochastic domination} again shows that the distribution of $f$ equals the distribution of $g+2k_0$, completing the proof of Theorem~\ref{thm:uniqueness_up_to_additive_constant}.

\section{Log-concavity and some consequences}\label{section: Log-concavity}
We further derive the log-concavity of the distribution of the height at a vertex.
    \begin{prop}\label{prop:log_concavity}
      Let $\Lambda\subset\Z^d$ be finite and $\tau:\extB\Lambda\to\Z$ be a homomorphism height function. Let $f$ be sampled from $\mu^{\Lambda,\tau}$. Then for any $v\in\Lambda$, the distribution of $f(v)$ is \emph{log-concave} in the sense that
      \begin{equation}\label{eq:log_concavity}
        \P(f(v) = m)^2\ge \P(f(v) = m+2k)\cdot\P(f(v)=m-2k)
      \end{equation}
      for integer $m$ of the same parity as $\|v\|_1$ and integer $k$.
      Consequently, if $f$ is sampled from an extremal Gibbs measure then \eqref{eq:log_concavity} remains true at any $v\in\Z^d$ and, in particular, $f(v)$ has finite moments of all orders.
    \end{prop}
    Log-concavity is known more generally for nearest-neighbor convex potentials from \cite[Lemma 8.2.4]{MR2251117} and for more general graphs (but prescribing boundary values at a single vertex rather than at a set) by Kahn \cite[Proposition 2.1]{MR1856513}. While the proofs are similar, we provide a proof for completeness.
\begin{proof}
Fix an integer $m$ with the same parity as $\|v\|_1$ and a positive integer $k$. For an integer $j$, let $H_j$ be the set of homomorphism height functions on $\Lambda^+$ which coincide with $\tau$ on $\extB\Lambda$ and equal $j$ at $v$. To prove~\eqref{eq:log_concavity} it suffices to build an injection from $H_{m+2k}\times H_{m-2k}$ to $H_m\times H_m$.

Let $h^+\in H_{m+2k}$ and $h^-\in H_{m-2k}$. Let $\Lambda'\subset\Z^d$ be the largest connected set containing $v$ on which $h^+>h^-+2k$. As $h^+ = h^-$ on $\extB\Lambda$ we must have that $\Lambda'\cap \extB\Lambda = \emptyset$. We may thus define homomorphism height functions $h, h'\in H_m$ by
\begin{align*}
  &h_{\Lambda'}:=(h^+-2k)_{\Lambda'},\quad h_{\Lambda\setminus{\Lambda'}}:=h^-_{\Lambda\setminus{\Lambda'}},\\
  &h'_{\Lambda'}:=(h^-+2k)_{\Lambda'},\quad h'_{\Lambda\setminus{\Lambda'}}:=h^+_{\Lambda\setminus{\Lambda'}}.
\end{align*}
Furthermore, $\Lambda'$ can be recovered from the pair $(h, h')$ as the largest connected set containing $v$ on which $h>h'-2k$. Thus the map $(h^+, h^-)\mapsto (h, h')$ is injective.
%
\end{proof}

Here is an immediate corollary of Theorem \ref{thm:ergodic_Gibbs_measures_are_extremal} and Proposition \ref{prop:log_concavity} proving that $\Ze^d$-ergodic Gibbs measures can be parametrized by the mean of the height at the origin.
\begin{cor}\label{corollary:mean}
Let $f$ and $g$ be sampled from $\Ze^d$-ergodic Gibbs measures. The mean of $f(\zero)$ is strictly greater than the mean of $g(\zero)$ if and only if $f$ strictly stochastically dominates $g$.
\end{cor}

\begin{proof} Let $f, g$ be sampled independently from their respective distributions. By Corollary~\ref{corollary:stochastic_domination}, it is sufficient to prove that if $f$ strictly stochastically dominates $g$ then the mean of $f(\zero)$ is strictly greater than the mean of $g(\zero)$. Consider $(\overline f, \overline g)$ as in Corollary \ref{cor:finite_cluster_equalizing} with $\eps\equiv 0$. Then $\overline f$ has the same distribution as $f$, $\overline g$ has the same distribution as $g$ and $\overline f\geq \overline g$ pointwise.  By Corollary~\ref{cor:stochastic domination} we have that  $\{\overline f>\overline g\}$ has an infinite connected component. In particular it follows that there exists $v\in \Ze^d$ such that $\overline{f}(v)$ strictly stochastically dominates $\overline{g}(v)$.
Since the measures are $\Ze^d$-translation invariant it follows that $f(\zero)$ strictly stochastically dominates $g(\zero)$. By the log concavity of the marginals (Proposition~\ref{prop:log_concavity}) it follows that the mean of $f(\zero)$ exists and is greater than the mean of $g(\zero)$.
\end{proof}

Lastly, we deduce that various notions of delocalization are equivalent in our case.

\begin{prop}
For odd integers $L\geq 1$, let $f_L$ be sampled from $\mu^{\Lambda(L),0}$. The following are equivalent:
\begin{enumerate}
\item
$\inf_{L}\P(|f_L(\zero)|=0)=0$.
\item
$\lim_{M\to \infty}\sup_L\P(|f_L(\zero)|>M)>0$.
\item
$\sup_L\P(|f_L(\zero)|>M)=1\quad$for all $M>0$.
\item
$  \sup_{L} \E(\exp(\alpha\cdot |f_L(\zero)|))=\infty\quad \text{for all $\alpha>0$}$.
\end{enumerate}
\end{prop}
\begin{proof}
Since the random variable $f_L(\zero)$ is symmetric and log-concave, $\P(f_L(\zero)=2n)$ is non-increasing as a function of integer $n\ge 0$. Thus~(1) implies~(3). It is obvious that~(3) implies~(2) and~(4). Let us see why each of~(2) and~(4) imply~(1).
Let  $\theta_L=\P(f_L(\zero)=0)$. If
$$2m_L=\min\left\{n\in \{2, 4, \ldots\}~:~\P(f_L(\zero)=n)<\frac{\theta_L}{2}\right\}$$
then $m_L<\frac{2}{\theta_L}$. Thus $\P(f_L(\zero)=2km_L)<\frac{\theta_L}{2^k}$ for all positive integers $k$ by the definitions of $\theta_L, m_L$ and the log-concavity of the distribution of $f_L(\zero)$. From here it is easy to see that both~(2) and~(4) individually imply that $\sup_L m_L=\infty$ whence $\inf_L \theta_L=0$. This completes the proof.
\end{proof}

\section{Discussion and Open questions}\label{section:open_questions}
    \begin{enumerate}
      \item In Theorem~\ref{thm:main} we prove that two-dimensional homomorphism height functions do not admit $\Ze^2$-translation-invariant Gibbs measures. It is however easy to construct Gibbs measures which are not translation-invariant, e.g., the ``frozen'' measure that assigns unit mass to $f \bigl( (x_1,x_2) \bigr)= x_1+x_2$. A probability measure on homomorphism height functions is called \emph{almost frozen} if its sample $f$ satisfies $$\P\left(\lim_{n\to\infty}\frac{f(u+ne)-f(u)}{n}=1\right)=1$$ for some fixed unit vector $e$ and all $u\in \Z^2$. There are uncountably many extremal Gibbs measures which are almost frozen. The following question naturally arises: Are all extremal Gibbs measures on two-dimensional homomorphism height functions almost frozen?

 \item Let $d\ge 3$. Let $f$ be sampled from a $\Ze^d$-ergodic Gibbs measure $\mu$ and fix a unit vector $e$. For all positive integers $n$ another $\Ze^d$-ergodic Gibbs measure $\nu$ can be constructed from it such that its sample $g$ satisfies $g(v)= f(v+ne)+n$ in distribution. Are these all the $\Ze^d$-ergodic Gibbs measures?

The answer to this question is positive for large $d$ by the results of~\cite{peled2018rigidity}. We sketch the main steps required to reach this conclusion: Any $3$-coloring of the $\Z^d$ lattice can be obtained by the modulo $3$ operation on homomorphism height functions. Further, for any two homomorphism height functions $f$ and $g$ which give the same $3$-coloring modulo $3$, there is an integer $k$ such that $f=g+6k$. It follows that any two $\Ze^d$-ergodic Gibbs measures which coincide modulo $3$ have a $\Ze^d$-invariant coupling such that its sample $(f, g)$ satisfies $f=g+6k$ where $k$ is now an integer-valued random variable. By Theorem~\ref{thm:ergodic_Gibbs_measures_are_extremal} and Proposition~\ref{prop:log_concavity}, we know that the mean of $f(v)$ and $g(v)$ exists for all $v\in \Z^d$. Thus the ergodic theorem implies that $k$ is a constant. In conclusion, if two $\Ze^d$-ergodic Gibbs measures on homomorphism height functions are equal modulo $3$ then they differ by addition of $6k$ for some integer $k$.

 We proved in Lemma \ref{lemma: Gibbs has maximal entropy} that $\Ze^d$-ergodic Gibbs measures on homomorphism height functions, when taken modulo $3$ are $\Ze^d$-ergodic measures of \emph{maximal entropy} for the $3$-coloring model. The $\Ze^d$-ergodic measures of maximal entropy for $3$-colorings have been characterized in~\cite{peled2018rigidity} for very high dimensions, with the following consequences:
 \begin{enumerate}
\item
 There are exactly $6$ such measures.
\item
Given a sample $f$ from a $\Ze^d$-ergodic Gibbs measure on homomorphism height function, the distribution of $v\mapsto f(v+ne)+n \pmod 3$ cycles through the $6$ measures as $n$ varies.
\end{enumerate}
The above facts imply the affirmative answer to the question in sufficiently high dimensions.
\item
Given graphs $G$ and $H$, we write $G \times H$ to denote the tensor product of $G$ with $H$ (in the tensor product $(g, h)$ is adjacent to $(g', h')$ if $g$ is adjacent to $g'$ and $h$ is adjacent to $h'$). It is easy to see that $\Z^2$ is isomorphic to each of  the two connected components of $\Z\times \Z$ (by the mapping $(i,j)\mapsto(i+j,i-j)$). It follows that our results for homomorphism height functions also hold for graph homomorphisms from $\Z^d$ to $\Z^2$. The following question naturally arises: Let $f_{L,m}$ be a uniformly sampled graph homomorphism from $\Lambda(L)^+\subset \Z^2$ to $\Z^m$, normalized by having $f_{L,m}$ map $\extB\Lambda(L)$ to the origin of $\Z^m$. For which $m$ does
      \begin{equation}
        \sup_{L} \var(|f_{L,m}(\zero)|)=\infty?
      \end{equation}

\item As described in the introduction, the model of uniformly sampled homomorphism height functions on $\Z^2$ is the $c=1$ case of the \emph{F-model}, which is a one-parameter sub-family of the $6$-vertex model indexed by the parameter $c>0$. In the F-model one weighs a homomorphism height function $f$ by $c^{-N(f)}$ where $N(f)$ is the number of diagonally-adjacent vertex pairs on which $f$ is unequal. It is conjectured that the F-model is delocalized if and only if $c\le 2$ (as a subset of the disordered regime of the $6$-vertex model). This conjecture has been verified when $c=1$, in this work, when $c\ge 2$ by a coupling with the random-cluster model~\cite{duminil2016discontinuity, glazmanpeled2018} and, very recently, for $c\in[\sqrt{2+\sqrt{2}},2]$ in~\cite{lis2020delocalization}.

\item The Gibbs measure uniqueness results in \cite{MR2251117} apply to convex, nearest-neighbor difference potentials on planar lattices. Using the standard parameters $a$,$b$ and $c$ (see, e.g.,~\cite{glazmanpeled2018}), the $6$-vertex model can be described using convex difference potentials when $\max(a,b) \leq c$ (this also corresponds to the parameters which make the model stochastically monotone as presented in \cite{LT2019}). Moreover, the graph representing the interactions between vertices is planar when either $a=c$ or $b=c$ which is a key requirement for using Theorem \ref{thm: no coexistence}. As a consequence, we believe that, using the cluster swapping technique mentioned in Section \ref{subsection: infinite clusters} and the argument presented in this paper, Theorem~\ref{thm:main} can be generalized to show that the $6$-vertex model delocalizes when the parameters satisfy either of the following two conditions: $a=c$ and $b \le c$ or $b=c$ and $a \leq c$. The uniformity of our model is not essential but the convexity of the potential, the planarity of the model and the fact that the parity of the height function changes on adjacent vertices is needed. These aspects are also used, together with additional arguments, in the proof of delocalisation in \cite{lammers_height_2020} (for certain integer-valued height functions on graphs with degree $3$ or less).
    \end{enumerate}

	\bibliographystyle{abbrv}
	\bibliography{3CB}
\end{document}